\documentclass[twoside,leqno,british,a4paper,12pt]
{amsart}

\usepackage{amsmath, amssymb, amsbsy, amsthm}

\renewcommand\theenumi{\arabic{enumi}}

\usepackage[T1]{fontenc}

\usepackage[british]{babel}
\usepackage[pagebackref,colorlinks=true]{hyperref}  

\begin{document}

\title[Associated Nijenhuis Tensors on Manifolds \ldots]
{Associated Nijenhuis Tensors on Manifolds with Almost Hypercomplex 
Structures and Metrics of Hermitian-Norden Type}
\author{MANCHO MANEV}

\address[1]{Department of Algebra and Geometry, Faculty of Mathematics and Informatics,
Plovdiv University\\ 236 Bulgaria Blvd, Plovdiv, 4027, Bulgaria}
\address[2]{Section of Mathematics and IT, Department of Pharmaceutical Science, Faculty of Pharmacy,
Medical University of Plovdiv \\ 15A Vasil Aprilov Blvd, Plovdiv, 4002, Bulgaria}

\email{mmanev@uni-plovdiv.bg}

\renewcommand\theenumi{(\roman{enumi})}
\newcommand{\f}{\varphi}
\newcommand{\n}{\nabla}
\newcommand{\lm}{\lambda}
\newcommand{\al}{\alpha}
\newcommand{\bt}{\beta}
\newcommand{\gm}{\gamma}
\newcommand{\om}{\omega}
\newcommand{\dt}{\delta}
\newcommand{\ta}{\theta}
\newcommand{\ep}{\varepsilon}
\newcommand{\ea}{\varepsilon_\alpha}
\newcommand{\eb}{\varepsilon_\beta}
\newcommand{\eg}{\varepsilon_\gamma}
\newcommand{\ba}{\bar\al}

\newcommand{\R}{\mathbb{R}}
\newcommand{\F}{\mathcal{F}}
\newcommand{\W}{\mathcal{W}}
\newcommand{\K}{\mathcal{K}}
\newcommand{\Id}{\mathrm{Id}}
\newcommand{\D}{\mathrm{d}}
\newcommand{\G}{\mathcal{G}}
\newcommand{\X}{\mathfrak{X}}
\newcommand{\LL}{\mathcal{L}}
\newcommand{\sx}{\mathop{\mathfrak{S}}\limits_{x,y,z}}
\newcommand{\s}{\mathfrak{S}}
\newcommand{\norm}[1]{\left\Vert#1\right\Vert ^2}
\newcommand{\M}{(M,\allowbreak\f_{\al},\allowbreak\xi_{\al},\allowbreak\eta_{\al},g)}
\newcommand{\fa}{(\f_{\al},\allowbreak\xi_{\al},\allowbreak\eta_{\al},g)}
\newcommand{\Lfa}{(L,\allowbreak\f_{\al},\allowbreak\xi_{\al},\allowbreak\eta_{\al},g)}
\newcommand{\Lfb}{(L,\allowbreak\f_{\bt},\allowbreak\xi_{\bt},\allowbreak\eta_{\bt},g)}
\newcommand{\Lfc}{(L,\allowbreak\f_{\gm},\allowbreak\xi_{\gm},\allowbreak\eta_{\gm},g)}
\newcommand{\thmref}[1]{The\-o\-rem~\ref{#1}}
\newcommand{\propref}[1]{Pro\-po\-si\-ti\-on~\ref{#1}}
\newcommand{\lemref}[1]{Lem\-ma~\ref{#1}}
\newcommand{\nT}{\norm{T}}
\newcommand{\nf}[1]{\norm{\n \f_{#1}}}
\newcommand{\corref}[1]{Corollary~\ref{#1}}
\newcommand{\ddt}{\frac{\D}{\D t}}

\newtheorem{thm}{Theorem}[section]
\newtheorem{lem}[thm]{Lemma}
\newtheorem{prop}[thm]{Proposition}
\newtheorem{cor}[thm]{Corollary}
\newtheorem{conv}[thm]{Convention}

\theoremstyle{definition}
\newtheorem{defn}{Definition}[section]
\newtheorem{rmk}[defn]{Remark}

\begin{abstract}
An associated Nijenhuis tensor of endomorphisms in the tangent bundle is introduced.
Special attention is paid to such tensors for an almost hypercomplex structure and the metric of Hermitian-Norden type.
There are studied relations between the six associated Nijenhuis tensors as well as their vanishing. It is given a geometric interpretation of the vanishing of these tensors as a necessary and sufficient condition for the existence of a unique connection with totally skew-symmetric torsion tensor.
Similar idea is used in the paper of T. Friedrich and S. Ivanov in Asian J. Math. (2002) for some other structures.
Finally, an example of a 4-dimensional manifold of the considered type with vanishing associated Nijenhuis tensors is given.
\end{abstract}



\maketitle
%


\section{Introduction}

The main problem of the differential geometry of manifolds with additional tensor structures and metrics is the studying of linear connections with respect to which the structures and the metrics are parallel. Such connections are also known as \emph{natural} for the considered manifolds. Natural connections with totally skew-symmetric torsion are an object of interest in theoretical and mathematical physics; for example, in string theory, in supergravity theories, in supersymmetric sigma models with Wess-Zumino term (see \cite{FrIv02} and references therein).

The object of study in the present work are almost hypercomplex manifolds.
The vanishing of the Nijenhuis tensors as conditions for in\-te\-gra\-bility of the manifold are long-known (\cite{YaAk}).

The most popular case of investigations of the considered manifold is the case of an almost hypercomplex manifold equipped with a Hermitian metric, i.e. the three almost complex structures act as isometries with respect to the metric in each tangent fibre (cf. \cite{AlMa}).

An interesting counterpart of the Hermitian metric is the Norden metric
, i.e. when the almost complex structure acts as an antiisometry with respect to the metric on the tangent bundle. This metric is pseu\-do-Riemannian of neutral signature and its associated (0,2)-tensor with respect to the almost complex structure is also a Norden metric unlike in the Hermitian case it is the K\"ahler form.
In another point of view, the pair of the associated Norden metrics are the real and the imaginary parts of a complex Riemannian metric (\cite{LeB83,Manin,GaIv92}).
The basic classifications of the almost complex manifolds with Hermitian metric and with Norden metric are given in \cite{GrHe} and \cite{GaBo}, respectively.

In the case of metrics of Hermitian and Norden types, the almost hypercomplex structure $(J_\al)$ ($\al=1,2,3$) and the metric $g$ generate one almost complex structure with pseudo-Riemannian Hermitian metric (for $\al=1$) and two almost complex structures with Norden metrics (for $\al=2;3$). Almost hypercomplex manifolds with metrics of Hermitian and Norden types are studied in \cite{GriManDim12,GriMan24,Man28,ManGri32}.

The geometry of a metric connection with totally skew-symmetric torsion preserving the almost hypercomplex structure with Hermitian metric is introduced in \cite{HoPa}.
A connection of such type, considered on almost hypercomplex manifolds with Hermitian and Norden metrics, is investigated in \cite{ManGri32}.

According to \cite{FrIv02}, on an almost Hermitian manifold, there exists a linear connection which has a totally skew-symmetric torsion preserving the almost complex structure and the Hermitian metric if and only if the Nijenhuis (0,3)-tensor is a 3-form. These manifolds form the class $\G_1$ of the cocalibrated structures, according to \cite{GrHe}.

In \cite{Mek08}, it is proved that there exists a linear connection with totally skew-symmetric torsion preserving the almost complex structure and the Norden metric if and only if the manifold is quasi-K\"ahlerian with Norden metric or it belongs to the class $\W_3$ in the classification in \cite{GaBo}. Moreover, this class is determined by the condition for vanishing of the associated Nijenhuis tensor (\cite{GaBo}, \cite{Man48}).

The goal of the present work is to introduce of an appropriate tensor on an almost hypercomplex manifold and to establish that its vanishing is a necessary and sufficient condition for the existence of a linear connection with totally skew-symmetric torsion preserving the almost hypercomplex structure and the metric of Hermitian-Norden type.


\section{The associated Nijenhuis tensors of endomorphisms on the tangent bundle of a pseudo-Riemannian manifold}\label{sect-endo}

Let us consider a differentiable manifold $M$.
As it is well known (\cite{KoNo}), the Nijenhuis tensor $[J,K]$ for arbitrary tensor fields $J$ and $K$ of type (1,1) on $M$ is a tensor field of type (1,2) defined by
\begin{equation}\label{N-JK}
\begin{array}{l}
  2[J,K](X,Y) = (JK+KJ)[X,Y]\\
  \phantom{2[J,K](X,Y) = }
  +[J X,K Y]-J[K X, Y]-J[X,K Y]\\
  \phantom{2[J,K](X,Y) = }
  +[K X,J Y]-K[J X, Y]-K[X,J Y],
\end{array}
\end{equation}
where $[X,Y]$ is the commutator for arbitrary differentiable vector fields $X$, $Y$ on $M$, i.e. $X,Y\in\mathfrak{X}(M)$.

As a consequence, the Nijenhuis tensor $[J,J]$ for arbitrary tensor field $J$ is determined by
\begin{equation}\label{NJ}
[J, J](X, Y)=\left[JX,JY\right]+J^2\left[X,Y\right]-J\left[JX,Y\right]-J\left[X,JY\right].
\end{equation}

If $J$ is an almost complex structure, the Nijenhuis tensor of $J$ is determined by the latter formula for $J^2=-I$, where $I$ stands for the identity.

Let us suppose that $g$ is a pseudo-Rieman\-nian metric on $M$ and $\n$ is the Levi-Civita connection of $g$.
Then, the symmetric braces $\{X,Y\}$ are defined as $\{X,Y\}=\n_X Y+\n_Y X$ by
\[
\begin{split}
g(\{X,Y\},Z)&=g(\nabla_XY+\nabla_YX,Z)\\
            &=Xg(Y,Z)+Yg(X,Z)-Zg(X,Y)\\
&\phantom{=}
+g([Z,X],Y)+g([Z,Y],X),\qquad
\end{split}
\]
where $X,Y,Z\in\mathfrak{X}(M)$ \cite{ManIv36,IvManMan45}.

Besides, in \cite{Man48}, we define the following symmetric
(1,2)-tensor in analogy to \eqref{NJ} by
\[
\{J ,J\}(X,Y)=\{JX,JY\}-\{X,Y\}-J\{JX,Y\}-J\{X,JY\},
\]
where the symmetric braces $\{X,Y\}$
are used instead of the antisymmetric brackets $[X,Y]=\nabla_XY-\nabla_YX$.
The tensor $\{J,J\}$ is called the  \emph{associated
Nijenhuis tensor} of the almost complex structure $J$ on $(M,g)$ (\cite{Man48}).

Bearing in mind \eqref{N-JK}, we define the following  tensor of type (1,2) for tensor fields $J$ and $K$ of type (1,1):
\begin{equation}\label{1.1}
\begin{array}{l}
  2\{J,K\}(X,Y) = (JK+KJ)\{X,Y\}\\
  \phantom{2\{J,K\}(X,Y) = }
  +\{J X,K Y\}-J\{K X, Y\}-J\{X,K Y\}\\
  \phantom{2\{J,K\}(X,Y) = }
  +\{K X,J Y\}-K\{J X, Y\}-K\{X,J Y\}.
\end{array}
\end{equation}
Obviously, it is symmetric with respect to tensor (1,1)-fields, i.e.
\begin{equation}\label{sym}
\{J,K\}(X,Y)=\{K,J\}(X,Y).
\end{equation}
Moreover, it is symmetric with respect to vector fields, i.e.
\[
\{J,K\}(X,Y)=\{J,K\}(Y,X).
\]

We call the tensor $\{J,K\}$ as the \emph{associated
Nijenhuis tensor} of $J$ and $K$, the tensors of type (1,1) on $(M,g)$.

In the case when $J$ or $K$ coincides with the identity $I$, then \eqref{1.1} yields that the corresponding associated
Nijenhuis tensor vanishes, i.e.
\begin{equation}\label{1.1a}
  \{J,I\}(X,Y) = 0,\qquad \{I,K\}(X,Y) = 0.
\end{equation}

For the case of two identical tensor (1,1)-fields, \eqref{1.1} implies
\[
\begin{array}{l}
  \{J,J\}(X,Y) = J^2\{X,Y\}+\{J X,J Y\}
  -J\{J X, Y\}-J\{X,J Y\}.
\end{array}
\]

Let $L$ be also a tensor field of type (1,1) and $S$ be a tensor field of type (1,2).
Further, we use the following notation of Fr\"olicher-Nijenhuis from \cite{FrNi}
\begin{gather}
  (S\barwedge L)(X,Y)=S(LX,Y)+S(X,LY),\label{1.4}\\
  (L\barwedge S)(X,Y)=L\bigl(S(X,Y)\bigr).\label{1.5}
\end{gather}
According to \cite{YaAk}, the following properties of the latter notation are valid
\begin{gather}
  \bigl(S\barwedge J\bigr)\barwedge K - \bigl(S\barwedge K\bigr)\barwedge J =
  S \barwedge JK - S \barwedge KJ,\label{1.7}\\
  (J\barwedge S)\barwedge K = J\barwedge (S\barwedge K).\label{1.8}
\end{gather}
%
%

\begin{lem}
For arbitrary endomorphisms $J$, $K$ and $L$ in $\mathfrak{X}(M)$ and the relevant associated Nijenhuis tensors,  the following identity holds
\begin{equation}\label{1.6}
\begin{array}{l}
  \{J,KL\}+\{K,JL\}
  =\{J,K\}\barwedge L + J \barwedge \{K,L\} + K \barwedge \{J,L\}.
\end{array}
\end{equation}
\end{lem}
\begin{proof}
Using \eqref{1.1}, \eqref{1.4} and \eqref{1.5}, we get consequently
\[
\begin{array}{l}
  \bigl(\{J,K\}\barwedge L\bigr) (X,Y)+ \bigl(J \barwedge \{K,L\}\bigr)(X,Y) + \bigl(K \barwedge \{J,L\}\bigr)(X,Y)\\
  =\{J,K\} (LX,Y)+ \{J,K\} (X,LY)+ J\bigl( \{K,L\}(X,Y)\bigr)\\
  \phantom{=} + K\bigl( \{J,L\}(X,Y)\bigr)\\
  =\frac12\bigl[(JK+KJ) \{LX,Y\} +\{JLX,KY\} +\{KLX,JY\}\\ \qquad
   -J\{KLX,Y\}  -K\{JLX,Y\}  -J\{LX,KY\} -K\{LX,JY\}\\ \qquad
   +(JK+KJ) \{X,LY\} +\{JX,KLY\} +\{KX,JLY\}\\ \qquad
   -J\{KX,LY\}  -K\{JX,LY\}  -J\{X,KLY\} -K\{X,JLY\}\\ \qquad
   +J
   (KL+LK)\{X,Y\}+J\{K X,L Y\}+J\{L X,K Y\}\\ \qquad
  -JK\{L X, Y\}-JL\{K X, Y\}-JK\{X,L Y\}-JL\{X,K Y\}
   \\
   \qquad
   +K
   (JL+LJ)\{X,Y\}+K\{J X,L Y\}+K\{L X,J Y\}\\ \qquad
  -KJ\{L X, Y\}-KL\{J X, Y\}-KJ\{X,L Y\}-KL\{X,J Y\}
   \bigl] \\
   =\frac12\bigl[\{JLX,KY\} +\{KLX,JY\}
   -J\{KLX,Y\}  -K\{JLX,Y\} \\ \qquad
    +\{JX,KLY\} +\{KX,JLY\}
   -J\{X,KLY\} -K\{X,JLY\}\\ \qquad
   +J
   (KL+LK)\{X,Y\}
  -JL\{K X, Y\}-JL\{X,K Y\}
   \\
   \qquad
   +K
   (JL+LJ)\{X,Y\}
  -KL\{J X, Y\}-KL\{X,J Y\}
   \bigl] \\
  = \{J,KL\}(X,Y)+\{K,JL\}(X,Y).
\end{array}
\]
\end{proof}

\section{The associated Nijenhuis tensors of the almost hypercomplex structure and a pseudo-Riemannian metric}\label{sect-hyper}

Let $(M,g)$ be a differentiable manifold of dimension $4n$ and a pseudo-Riemannian metric $g$ on it. The manifold $M$ admits an almost hypercomplex structure  $(J_1,J_2,J_3)$, i.e. a triple of almost complex structures satisfying the properties:
\begin{equation}\label{J123} %
J_\al=J_\bt\circ J_\gm=-J_\gm\circ J_\bt, \qquad
J_\al^2=-I%
\end{equation} %
for all cyclic permutations $(\al, \bt, \gm)$ of $(1,2,3)$, where $I$
denotes the identity (\cite{Boy}, \cite{AlMa}).

\subsection{Relations between the associated Nijenhuis tensors }

The presence of three almost complex structures implies the existence of six associated Nijenhuis tensors -- three for each almost complex structure and three for each pair of almost complex structures --  namely
$\{J_1,J_1\}$, $\{J_2,J_2\}$, $\{J_3,J_3\}$,
$\{J_1,J_2\}$, $\{J_2,J_3\}$, $\{J_3,J_1\}$. In the present section we examine some relations between them, which we use later.

\begin{lem}\label{lem:3.1}
The following relations between the associated Nijenhuis tensors of an almost hypercomplex manifold are valid:
\begin{gather}
  \{J_3,J_1\} =\frac12 \{J_1,J_1\}\barwedge J_2 + J_1 \barwedge \{J_1,J_2\},\label{2.2}
\\
  \{J_3,J_1\}
  =-\{J_1,J_2\}\barwedge J_1 - J_1 \barwedge \{J_1,J_2\} - J_2 \barwedge \{J_1,J_1\},\label{2.3}
\\
 \begin{split}\label{2.5}
J_2 \barwedge \{J_1,J_1\}&+\frac12 \{J_1,J_1\}\barwedge J_2\\
&
+ 2 J_1 \barwedge \{J_1,J_2\}  +\{J_1,J_2\}\barwedge J_1
=0,
\end{split}
\\
  \{J_2,J_3\}
  =-\frac12 \{J_2,J_2\}\barwedge J_1 - J_2 \barwedge \{J_1,J_2\},\label{2.6}
\\
  \{J_2,J_3\}
  =J_1 \barwedge \{J_2,J_2\} + \{J_1,J_2\}\barwedge J_2 + J_2 \barwedge \{J_1,J_2\},\label{2.7}
\\
 \begin{split}\label{2.9}
J_1 \barwedge \{J_2,J_2\} &+\frac12 \{J_2,J_2\}\barwedge J_1\\
&
  + \{J_1,J_2\}\barwedge J_2 + 2 J_2 \barwedge \{J_1,J_2\}=0,
\end{split}
\\
\begin{split}\label{2.10}
  \{J_3,J_3\}-\{J_1,J_1\}
  =\{J_3,J_1\}\barwedge J_2 &+ J_3 \barwedge \{J_1,J_2\}\\
  &+ J_1 \barwedge \{J_2,J_3\},
\end{split}
\\
\begin{split}\label{2.12}
  \{J_3,J_3\}=\frac12\bigl(\{J_1,J_1\}
  &+\{J_3,J_1\}\barwedge J_2 - J_2 \barwedge \{J_3,J_1\}\phantom{\bigr),}\\
  &-\{J_2,J_3\}\barwedge J_1+ J_1 \barwedge \{J_2,J_3\}\bigr),
\end{split}
\\
\begin{split}\label{2.13}
\{J_1,J_1\}&-\{J_2,J_2\}+\{J_3,J_1\}\barwedge J_2 + J_2 \barwedge \{J_3,J_1\}\phantom{=0.\ }\\
&+ 2J_3 \barwedge \{J_1,J_2\}
+\{J_2,J_3\}\barwedge J_1+ J_1 \barwedge \{J_2,J_3\}=0,
\end{split}
\\
    \{J_2,J_2\}\barwedge J_2 =- 2 J_2 \barwedge \{J_2,J_2\}.\label{2.15}
\end{gather}
\end{lem}

\begin{proof}
The identity \eqref{1.6} for $J=K=J_1$ and $L=J_2$ has the form in \eqref{2.2}, because of $J_1J_2=J_3$ from \eqref{J123} and \eqref{sym}.

On the other hand,  \eqref{1.6} for $J=J_2$ and $K=L=J_1$ implies
\[
  \{J_2,J_1J_1\}+\{J_1,J_2J_1\}
  =\{J_2,J_1\}\barwedge J_1 + J_2 \barwedge \{J_1,J_1\}
  + J_1 \barwedge \{J_2,J_1\}.
\]
Next, applying $J_1^2=-I$, $J_2J_1=-J_3$, which are corollaries of \eqref{J123}, as well as \eqref{sym} and \eqref{1.1a}, we obtain
\eqref{2.3}.

By virtue of \eqref{2.2} and \eqref{2.3} we get 
\eqref{2.5}.

Next, setting $J=K=J_2$ and $L=J_1$ in \eqref{1.6}, we find
\[
  \{J_2,J_2J_1\}+\{J_2,J_2J_1\}
  =\{J_2,J_2\}\barwedge J_1 + J_2 \barwedge \{J_2,J_1\}
   + J_2 \barwedge \{J_2,J_1\},
\]
which is equivalent to \eqref{2.6}, taking into account \eqref{sym} and $J_2J_1=-J_3$.

On the other hand, setting $J=J_1$ and $K=L=J_2$ in \eqref{1.6}, we find
\[
  \{J_1,J_2J_2\}+\{J_2,J_1J_2\}
  =\{J_1,J_2\}\barwedge J_2 + J_1 \barwedge \{J_2,J_2\}
  + J_2 \barwedge \{J_1,J_2\}
\]
and applying \eqref{1.1a} and $J_1J_2=J_3$, we have \eqref{2.7}.

By virtue of \eqref{2.6} and \eqref{2.7} we get 
\eqref{2.9}.

Now, setting $J=J_3$, $K=J_1$ and $L=J_2$ in \eqref{1.6} and applying $J_1J_2=J_3$, $J_3J_2=-J_1$ and \eqref{sym}, we have
\eqref{2.10}.

The equality \eqref{2.10} and the resulting equality from it by the substitutions $J_3$, $J_1$, $J_2$ for $J_1$, $J_2$, $J_3$, respectively, imply \eqref{2.12} and \eqref{2.13}.



At the end, the identity \eqref{1.6} for $J=K=L=J_2$ implies
\eqref{2.15} because of \eqref{1.1a}.
\end{proof}

\subsection{Vanishing of the associated Nijenhuis tensors}

In this subsection we focus our study on sufficient conditions for the vanishing of all associated Nijenhuis tensors.

As proof steps of the relevant main theorem, we need to prove a series of lemmas.

\begin{lem}\label{thm:2}
If $\{J_1,J_1\}=0$ and $\{J_2,J_2\}=0$, then $\{J_1,J_2\}=0$, $\{J_2,J_3\}=0$, $\{J_3,J_1\}=0$ and $\{J_3,J_3\}=0$.
\end{lem}
\begin{proof}
The formulae \eqref{2.2} and \eqref{2.5}, because of $\{J_1,J_1\}=0$, imply respectively
\begin{gather}
  \{J_3,J_1\} = J_1 \barwedge \{J_1,J_2\},\label{3.1}\\
\{J_1,J_2\}\barwedge J_1
=-2 J_1 \barwedge \{J_1,J_2\}.\label{3.2}
\end{gather}

Similarly, since $\{J_2,J_2\}=0$, the equalities \eqref{2.6} and \eqref{2.9} take the corresponding form
\begin{gather}
  \{J_2,J_3\}
  =- J_2 \barwedge \{J_1,J_2\},\label{3.3}\\
\{J_1,J_2\}\barwedge J_2 =- 2 J_2 \barwedge \{J_1,J_2\}.\label{3.4}
\end{gather}

We set $\{J_1,J_1\}=0$, $\{J_2,J_2\}=0$, \eqref{3.1} and \eqref{3.3} in  \eqref{2.13} and obtain
\[
\begin{array}{l}
2J_3 \barwedge \{J_1,J_2\}
+\bigl(J_1 \barwedge \{J_1,J_2\}\bigr)\barwedge J_2 + J_2 \barwedge \bigl(J_1 \barwedge \{J_1,J_2\}\bigr)\\
  \phantom{2J_3 \barwedge \{J_1,J_2\}}
-\bigl(J_2 \barwedge \{J_1,J_2\}\bigr)\barwedge J_1- J_1 \barwedge \bigl(J_2 \barwedge \{J_1,J_2\}\bigr)=0.
\end{array}
\]
The latter equality is equivalent to
\begin{equation}\label{3.5}
\begin{array}{l}
\bigl(J_1 \barwedge \{J_1,J_2\}\bigr)\barwedge J_2
=\bigl(J_2 \barwedge \{J_1,J_2\}\bigr)\barwedge J_1,
\end{array}
\end{equation}
because of the following corollaries of \eqref{1.5} and the identities $J_3=J_1J_2=-J_2J_1$ from \eqref{J123}
\begin{equation}\label{3.5'}
\begin{array}{l}
J_2 \barwedge \bigl(J_1 \barwedge \{J_1,J_2\}\bigr)=
-J_1 \barwedge \bigl(J_2 \barwedge \{J_1,J_2\}\bigr)=
-J_3 \barwedge \{J_1,J_2\}.
\end{array}
\end{equation}

According to \eqref{1.8}, \eqref{3.2}, \eqref{3.4} and \eqref{3.5'}, the equality \eqref{3.5} yields
\[
\begin{array}{l}
0=
\bigl(J_1 \barwedge \{J_1,J_2\}\bigr)\barwedge J_2
-\bigl(J_2 \barwedge \{J_1,J_2\}\bigr)\barwedge J_1\\
\phantom{0}
=
J_1 \barwedge \bigl(\{J_1,J_2\}\barwedge J_2\bigr)
-J_2 \barwedge \bigl(\{J_1,J_2\}\barwedge J_1\bigr)\\
\phantom{0}
=
-2J_1 \barwedge \bigl(J_2\barwedge \{J_1,J_2\}\bigr)
+2J_2 \barwedge \bigl(J_1\barwedge \{J_1,J_2\}\bigr)\\
\phantom{0}
=
-4J_3 \barwedge \{J_1,J_2\},
\end{array}
\]
i.e. it is valid
\[
\begin{array}{l}
J_3 \barwedge \{J_1,J_2\}=0.
\end{array}
\]
Therefore, because of $J_3^2=-I$, we get
\[
\{J_1,J_2\}=0.
\]

Next, \eqref{3.1} and \eqref{3.3} imply
$\{J_3,J_1\}=0$ and $\{J_2,J_3\}=0$, respectively.
Finally, since $\{J_1,J_1\}$, $\{J_2,J_2\}$, $\{J_2,J_3\}$ and $\{J_3,J_1\}$ vanish, the relation
\eqref{2.12} yields $\{J_3,J_3\}=0$.
\end{proof}

\begin{lem}\label{thm:4}
If $\{J_1,J_1\}=0$ and $\{J_1,J_2\}=0$, then $\{J_2,J_2\}=0$, $\{J_3,J_3\}=0$, $\{J_2,J_3\}=0$, $\{J_3,J_1\}=0$.
\end{lem}
\begin{proof}
Setting $J=K=J_2$ and $L=J_3$ in \eqref{1.6}, using $J_2J_3=J_1$ and \eqref{sym}, we obtain
\[
  \{J_1,J_2\}
  =\frac12 \{J_2,J_2\}\barwedge J_3 + J_2 \barwedge \{J_2,J_3\}.
\]
Since $\{J_1,J_2\}=0$, we get
\begin{equation}\label{3.7}
  \{J_2,J_2\}\barwedge J_3 =-2 J_2 \barwedge \{J_2,J_3\}.
\end{equation}

The condition $\{J_1,J_2\}=0$ and \eqref{2.7} imply
\[
  \{J_2,J_3\}
  =J_1 \barwedge \{J_2,J_2\}.
\]
The latter equality and \eqref{3.7}, using $J_2J_1=-J_3$, yield
\[
  \{J_2,J_2\}\barwedge J_3 =-2 J_2 \barwedge \bigl(J_1 \barwedge \{J_2,J_2\}\bigr)
  =2 J_3\barwedge \{J_2,J_2\},
\]
that is
\begin{equation}\label{3.8}
  \{J_2,J_2\}\barwedge J_3   =2 J_3\barwedge \{J_2,J_2\}.
\end{equation}

On the other hand, setting $S=\{J_2,J_2\}$, $J=J_1$ and $K=J_2$ in \eqref{1.7} and using $J_1J_2=-J_2J_1=J_3$, we obtain
\begin{equation}\label{3.9}
\begin{array}{l}
  2\{J_2,J_2\} \barwedge J_3=
  \bigl(\{J_2,J_2\}\barwedge J_1\bigr)\barwedge J_2 - \bigl(\{J_2,J_2\}\barwedge J_2\bigr)\barwedge J_1.
\end{array}
\end{equation}

However, since $\{J_1,J_2\}=0$, then \eqref{2.9} and \eqref{2.10} imply respectively
\begin{equation}\label{3.10}
 \begin{array}{l}
\{J_2,J_2\}\barwedge J_1
=-2J_1 \barwedge \{J_2,J_2\}.
\end{array}
\end{equation}

Now, substituting \eqref{2.15} and \eqref{3.10} into \eqref{3.9}, we obtain
\[
\begin{array}{l}
  \{J_2,J_2\} \barwedge J_3=
  -\bigl(J_1 \barwedge \{J_2,J_2\}\bigr)\barwedge J_2 + \bigl(J_2 \barwedge \{J_2,J_2\}\bigr)\barwedge J_1
\end{array}
\]
and applying \eqref{1.8}, we have
\[
\begin{array}{l}
  \{J_2,J_2\} \barwedge J_3=
  -J_1 \barwedge \bigl(\{J_2,J_2\}\barwedge J_2\bigr) +J_2 \barwedge \bigl(\{J_2,J_2\}\barwedge J_1\bigr).
\end{array}
\]
In the latter equality, applying \eqref{2.15}, \eqref{3.10} and $J_1J_2=-J_2J_1=J_3$, we get
\[
\begin{array}{l}
  \{J_2,J_2\} \barwedge J_3=
  2J_1 \barwedge \bigl(J_2\barwedge \{J_2,J_2\}\bigr) -2J_2 \barwedge \bigl(J_1\barwedge \{J_2,J_2\}\bigr)\\
  \phantom{\{J_2,J_2\} \barwedge J_3}
  =4J_3 \barwedge \{J_2,J_2\},
\end{array}
\]
that is
\begin{equation}\label{3.11}
  \{J_2,J_2\} \barwedge J_3  =4J_3 \barwedge \{J_2,J_2\}.
  \end{equation}

Comparing \eqref{3.8} and \eqref{3.11}, we conclude that
\[
  J_3 \barwedge \{J_2,J_2\}=0,
\]
which is equivalent to
\[
\{J_2,J_2\}=0
\]
by virtue of $J_3^2=-I$. This completes the proof of the first assertion in the lemma.

Combining it with \lemref{thm:2}, 
we establish the truthfulness of the whole lemma. 
\end{proof}

\begin{lem}\label{thm:6}
If $\{J_1,J_2\}=0$ and $\{J_3,J_1\}=0$, then $\{J_1,J_1\}=0$, $\{J_2,J_2\}=0$, $\{J_3,J_3\}=0$, $\{J_2,J_3\}=0$.
\end{lem}
\begin{proof}
From \eqref{2.3} and the vanishing of $\{J_1,J_2\}$ and $\{J_3,J_1\}$, we get
\[
  J_2 \barwedge \{J_1,J_1\}=0,
\]
which is equivalent to
\[
  \{J_1,J_1\}=0.
\]

Now, combining the latter assertion and \lemref{thm:4}, 
we have the validity of the present lemma. 
\end{proof}

\begin{lem}\label{thm:8}
If $\{J_1,J_1\}=0$ and $\{J_2,J_3\}=0$, then $\{J_2,J_2\}=0$, $\{J_3,J_3\}=0$, $\{J_1,J_2\}=0$, $\{J_3,J_1\}=0$.
\end{lem}
\begin{proof}
Firstly, from \eqref{2.2} and $\{J_1,J_1\}=\{J_2,J_3\}=0$, we obtain
\begin{equation}\label{3.12}
  \{J_3,J_1\} = J_1 \barwedge \{J_1,J_2\}.
\end{equation}

Secondly, setting $J=K=J_2$ and $L=J_3$ in \eqref{1.6} and using $J_2J_3=J_1$, $\{J_1,J_1\}=\{J_2,J_3\}=0$ and \eqref{sym}, we get
\[
  2\{J_1,J_2\}
  =\{J_2,J_2\}\barwedge J_3.
\]

On the other hand, setting $J=L=J_2$ and $K=J_3$ in \eqref{1.6}, using the assumptions, as well as $J_2^2=-I$, $J_3J_2=-J_1$  and \eqref{1.1a},
we find
\[
  \{J_1,J_2\}
  = -J_3 \barwedge \{J_2,J_2\},
\]
which, because of $J_3^2=-I$, is equivalent to
  \begin{equation}\label{3.14}
  \{J_2,J_2\}= J_3\barwedge \{J_1,J_2\}.
  \end{equation}

Moreover, the formula \eqref{2.13} and $\{J_1,J_1\}=\{J_2,J_3\}=0$ yield
\[
2J_3 \barwedge \{J_1,J_2\}
+\{J_3,J_1\}\barwedge J_2 + J_2 \barwedge \{J_3,J_1\}-\{J_2,J_2\}=0.
\]
The latter equality, \eqref{3.12} and \eqref{3.14}, using $J_2J_1=-J_3$, \eqref{1.5} and \eqref{1.8}, imply
\[
J_1 \barwedge \bigl(\{J_1,J_2\}\barwedge J_2\bigr) =0
\]
and therefore
\begin{equation}\label{3.15}
\{J_1,J_2\}\barwedge J_2 =0.
\end{equation}

By virtue of \eqref{3.14}, \eqref{3.15} and \eqref{1.8}, we get
\begin{equation}\label{3.16}
  \{J_2,J_2\}\barwedge J_2= 0.
\end{equation}

On the other hand, \eqref{2.15} and \eqref{3.16} imply
\[
J_2 \barwedge \{J_2,J_2\}=0,
\]
which is equivalent to
\[
\{J_2,J_2\}=0
\]
and the first assertion in the present lemma is proved.

Combining it with \lemref{thm:2}, 
we obtain the validity of the rest equalities. 
\end{proof}

Now, we are ready to prove the main theorem in the present section.

\begin{thm}\label{thm:9}
If two of the six associated Nijenhuis tensors
\[
\{J_1,J_1\},\quad \{J_2,J_2\},\quad \{J_3,J_3\},\quad
\{J_1,J_2\},\quad \{J_2,J_3\},\quad \{J_3,J_1\}
\]
vanish, then the others also vanish.
\end{thm}
\begin{proof}
The truthfulness of this theorem follows from \lemref{thm:2}, \lemref{thm:4}, \lemref{thm:6} and \lemref{thm:8}.
\end{proof}


\section{Almost hypercomplex manifolds with metrics of Hermitian-Norden type having vanishing associated Nijenhuis tensors}\label{sect-assNije}

Let $g$ be a pseudo-Riemannian metric on an almost hypercomplex manifold $(M,J_1,J_2,J_3)$ defined by
\begin{equation}\label{gJJ} %
g(x,y)= \ea g(J_\al x,J_\al y),
\end{equation} %
where %
\[
 \ea=
\begin{cases}
\begin{array}{ll}
1, \quad & \al=1;\\
-1, \quad & \al=2,3.
\end{array}
\end{cases}
\]
Then, we call that the almost hypercomplex manifold is equipped with an HN-metric structure (HN is an abbreviation for Hermitian-Norden).
Namely, the metric $g$ is Hermitian for $\al=1$, whereas $g$ is a Norden metric in the cases  $\al=2$ and $\al=3$ (\cite{GriMan24}, \cite{ManGri32}).

Let us consider $(M,J_1,g)$ in 
the class $\G_1$ (the so-called cocalibrated manifolds with Hermitian
metric), 
according to the classification in \cite{GrHe}, as well as $(M,J_\al,g)$,
$\al=2;3$, in 
the class $\W_3$ (the so-called quasi K\"ahler manifolds with Norden
metric), according to the classification in \cite{GaBo}.
The mentioned classes are determined in
terms of the fundamental tensors $F_{\al}$, defined by $F_{\al}(x,y,z)=g((\n_xJ_\al)y,z)$ as follows:
\begin{eqnarray}
\nonumber
    \G_1(J_1)&:&\
    F_1 (x,x,z) = F_1 (J_1 x,J_1 x,z), \\
\nonumber
    \W_3(J_{\al})&:&\
    F_{\al} (x,y,z)+F_{\al} (y,z,x)+F_{\al} (z,x,y)=0, 
    \quad \al=2;3.
\end{eqnarray}

It is known from \cite{GrHe} that the class
$\G_1$ of almost Hermitian manifolds
$(M,J_1,g)$ exists in general form when the dimension of $M$ is at
least 6. For dimension 4, $\G_1$ is restricted to its subclass $\W_4$, the class of locally conformally equivalent manifolds to K\"ahler manifolds.
According to \cite{GaBo}, the lowest dimension for almost Norden manifolds in the class $\W_3$ is 4.
Thus, the almost hypercomplex HN-metric manifold belonging to the classes
$\G_1(J_1)$, $\W_3(J_2)$, $\W_3(J_3)$ exists in general form when $\dim{M}=4n\geq 8$.

The tensors $F_{\al}$ $(\al=1,2,3)$ have the following basic properties caused by the structures
\begin{equation}\label{Fa-prop}
  F_{\al}(x,y,z)=-\ea F_{\al}(x,z,y)=-\ea F_{\al}(x,J_{\al}y,J_{\al}z).
\end{equation}

Let the corresponding tensors of type (0,3) with respect to $g$ be denoted by
\[
\begin{array}{l}
[J_\al,J_\al](x,y,z)=g([J_\al,J_\al](x,y),z),\\
\{J_\al,J_\al\}(x,y,z)=g(\{J_\al,J_\al\}(x,y),z).
\end{array}
\]

The tensors $[J_\al,J_\al]$ and $\{J_\al,J_\al\}$ can be expressed by $F_{\al}$, as follows:
\begin{gather}
\begin{array}{l}
[J_\al,J_\al](x,y,z)=
F_{\al}(J_{\al} x,y,z)+\ea F_{\al}(x,y,J_{\al} z)
\label{enu}\\
\phantom{[J_\al,J_\al](x,y,z)=}
-F_{\al}(J_{\al} y,x,z)-\ea F_{\al}(y,x,J_{\al} z),
\end{array}\\
\begin{array}{l}
\{J_\al,J_\al\}(x,y,z)=
F_{\al}(J_{\al} x,y,z)+\ea F_{\al}(x,y,J_{\al} z)\\
\phantom{\{J_\al,J_\al\}(x,y,z)=}
+F_{\al}(J_{\al} y,x,z)+\ea F_{\al}(y,x,J_{\al}z).\label{enhat}
\end{array}
\end{gather}
In the case $\al=2;3$, the latter formulae are given in \cite{GaBo}, as $\{J_\al,J_\al\}$ coincides with the tensor $\widetilde N$
introduced there by an equivalent equality of \eqref{enhat}.

In \cite{GrHe}, it is given an equivalent definition of $\G_1$ by the condition the Nijenhuis tensor $[J_1,J_1](x,y,z)$ to be a 3-form, i.e.
\begin{equation}\label{G1-3f}
  \G_1(J_1):\quad [J_1,J_1](x,y,z)=-[J_1,J_1](x,z,y).
\end{equation}

\begin{thm}\label{thm:NN=Nhat}
For the Nijenhuis tensor and the associated Nijenhuis tensor of a pseudo-Riemannian almost Hermitian manifold $(M,J_1,g)$. The following assertions are equivalent:
\begin{enumerate}
  \item $\{J_1,J_1\}$ vanishes;
  \item $[J_1,J_1]$ is a 3-form;
  \item There exists a linear connection with totally skew-symmetric torsion preserving the structure $(J_1,g)$ and this connection is unique.
\end{enumerate}
\end{thm}

\begin{proof}
First it has to prove the following relation
\begin{equation}\label{NN=Nhat}
\{J_1,J_1\}(x,y,z)=[J_1,J_1](z,x,y)+[J_1,J_1](z,y,x).
\end{equation}
We compute the right hand side of \eqref{NN=Nhat} using \eqref{enu}. Applying \eqref{Fa-prop} and their consequence
\begin{equation}\label{FffF1}
  F_1(x,y,J_1z)=F_1(x,J_1y,z),
\end{equation}
we obtain
\[
\begin{array}{l}
[J_1,J_1](z,x,y)+[J_1,J_1](z,y,x)=
-F_1(J_1x,z,y)-F_1(x,z,J_1y)\\
\phantom{[J_1,J_1](z,x,y)+[J_1,J_1](z,y,x)=}
-F_1(J_1y,z,x)-F_1(y,z,J_1x).
\end{array}
\]
Using again \eqref{FffF1} and the first equality in \eqref{Fa-prop}, we establish that the right hand side of the latter equality is equal to $\{J_1,J_1\}(x,y,z)$, according to \eqref{enhat} for $\al=1$.
Thus, \eqref{NN=Nhat} is valid.

The relation \eqref{NN=Nhat} implies the equivalence between (i) and (ii), whereas
the equivalence between (ii) and (iii) follows from Theorem 10.1 in \cite{FrIv02}.
\end{proof}

\begin{thm}\label{thm:HKT}
The vanishing of two of the six associated Nijenhuis tensors for an almost hypercomplex structure is a necessary and suffi\-cient condition for existing a linear connection with totally skew-sym\-met\-ric torsion preserving the almost hypercomplex HN-metric structure.
\end{thm}

\begin{proof}

From \thmref{thm:NN=Nhat} and \eqref{G1-3f}, we obtain that 
the manifolds in the class $\G_1(J_1)$ are characterized by the condition $\{J_1,J_1\}=0$.

For almost Norden manifolds it is known from \cite{GaBo} that the manifolds in the class $\W_3(J_2)$ (respectively, $\W_3(J_3)$) are characterized by the condition $\{J_2,J_2\}=0$ (respectively, $\{J_3,J_3\}=0$).

In \cite{ManGri32}, it is proved that if
an almost hypercomplex HN-metric manifold  is in $\W_3(J_2)$ and $\W_3(J_3)$, then it
belongs to $\G_1(J_1)$.

From \thmref{thm:9} we have immediately that
if two of associated Nijenhuis tensors $\{J_1,J_1\}$, $\{J_2,J_2\}$, $\{J_3,J_3\}$ vanish, then the third one also vanishes.
%
Thus, we establish that
if an almost hypercomplex HN-metric manifold belongs to two of the classes
$\G_1(J_1)$, $\W_3(J_2)$, $\W_3(J_3)$
with respect to the corresponding almost complex structures, then it belongs also to the third class.

By virtue of \thmref{thm:9}, \thmref{thm:NN=Nhat}  and the comments above, we obtain the validity of the statement.
\end{proof}

\section{A 4-dimensional example}\label{sect-exa}

In \cite{GriMan24}, it is considered a connected Lie group $G$ with a corresponding Lie algebra $\mathfrak{g}$, 
determined by the following conditions for the global basis of left invariant vector fields $\{X_1,X_2,X_3,X_4\}$:
\[
\begin{array}{ll}
[X_1,X_3]= \lambda_2 X_2+\lambda_4 X_4,\qquad &
[X_2,X_4]= \lambda_1 X_1+\lambda_3 X_3, \\[0pt]
[X_3,X_2]= \lambda_2 X_1+\lambda_3 X_4,\qquad &
[X_4,X_3]= \lambda_4 X_1-\lambda_3 X_2, \\[0pt]
[X_4,X_1]= \lambda_1 X_2+\lambda_4 X_3,\qquad &
[X_1,X_2]= \lambda_2 X_3-\lambda_1 X_4,
\end{array}
\]
where $\lambda_i\in \mathbb{R}$ $(i=1,2,3,4)$ and
$(\lambda_1,\lambda_2,\lambda_3,\lambda_4)\neq (0,0,0,0)$.
The pseu\-do-Riemannian metric $g$ is defined by 
\begin{equation}\label{g}
\begin{array}{l}
  g(X_1,X_1)=g(X_2,X_2)=-g(X_3,X_3)=-g(X_4,X_4)=1, \\
  g(X_i,X_j)=0,\quad i\neq j.
\end{array}
\end{equation}
There, it is introduced an almost hypercomplex structure $(J_1,J_2,J_3)$ on $G$ as follows:
\begin{equation}\label{J123-ex}
\begin{array}{llll}
J_1X_1=X_2,\; & J_1X_2=-X_1,\; & J_1X_3=-X_4,\; & J_1X_4=X_3; \\
J_2X_1=X_3,\; & J_2X_2=X_4,\; & J_2X_3=-X_1,\; & J_2X_4=-X_2; \\
J_3X_1=-X_4,\; & J_3X_2=X_3,\; & J_3X_3=-X_2,\; & J_3X_4=X_1
\end{array}
\end{equation}
and then \eqref{J123} are valid.

In \cite{GrMaMe1}, it is constructed the manifold $(G,J_2,g)$ as an example of a 4-dimensional quasi-K\"ahler manifold with Norden metric. 

The conditions \eqref{g} and \eqref{J123-ex} imply the properties \eqref{gJJ} and therefore the introduced structure on $G$ is an almost hypercomplex HN-metric structure. Hence, it follows that the constructed manifold is
an almost hypercomplex HN-metric manifold.
Moreover, in \cite{GriMan24}, it is shown that this manifold belongs to basic classes $\W_4(J_1)$, $\W_3(J_2)$, $\W_3(J_3)$ with respect to the corresponding almost complex structures.

Bearing in mind the previous two sections, here we consider an example of a 4-dimen\-sion\-al manifold with vanishing associated Nijenhuis tensors for the almost hypercomplex structure. Therefore, in view of \thmref{thm:HKT}, there exists a linear connection with totally skew-sym\-met\-ric torsion preserving the almost hypercomplex HN-metric structure.

\section*{Acknowledgments}
The author wishes to thank Professor Stefan Ivanov for useful comments about this work.

The present paper was partially supported by project NI15-FMI-004 of the
Scientific Research Fund at the University of Plovdiv.


\end{document}